\documentclass{amsart}
\usepackage{amsmath}
\usepackage{amssymb}
\usepackage{amsthm}
\usepackage{enumerate}
\usepackage[pdftex]{graphicx}
\usepackage{caption}
\theoremstyle{definition}
\newtheorem{definition}{Definition}[section]
\theoremstyle{plain}
\newtheorem{lemma}[definition]{Lemma}
\newtheorem{theorem}[definition]{Theorem}

\newtheorem{proposition}[definition]{Proposition}
\newtheorem{corollary}[definition]{Corollary}
\theoremstyle{remark}
\newtheorem{remark}[definition]{Remark}

\newtheorem{example}[definition]{Example}

\makeatletter
\@namedef{subjclassname@2020}{
  \textup{2020} Mathematics Subject Classification}
\makeatother

\newcommand{\mycl}{\operatorname{cl}}
\newcommand{\myint}{\operatorname{int}}

\newcommand{\mylns}{\operatorname{lns}}

\begin{document}
\title[Tameness of locally o-minimal structures]{Tameness of definably complete locally o-minimal structures and definable bounded multiplication}
\author[M. Fujita]{Masato Fujita}
\address{Department of Liberal Arts,
Japan Coast Guard Academy,
5-1 Wakaba-cho, Kure, Hiroshima 737-8512, Japan}
\email{fujita.masato.p34@kyoto-u.jp}

\author[T. Kawakami]{Tomohiro Kawakami}
\address{Department of Mathematics,
Wakayama University,
Wakayama, 640-8510, Japan}
\email{kawa0726@gmail.com}

\author[W. Komine]{Wataru Komine}
\address{Institute of Mathematics, University of Tsukuba, Ibaraki, 305-8571, Japan}
\email{12tarous@gmail.com}
\begin{abstract}
We first show that the projection image of a discrete definable set is again discrete for an arbitrary definably complete locally o-minimal structure. 
This fact together with the results in a previous paper implies tame dimension theory and decomposition theorem into good-shaped definable subsets called quasi-special submanifolds.

Using this fact, in the latter part of this paper, we investigate definably complete locally o-minimal expansions of ordered groups when the restriction of multiplication to an arbitrary  bounded open box is definable.
Similarly to o-minimal expansions of ordered fields, {\L}ojasiewicz's inequality, Tietze extension theorem and affiness of psudo-definable spaces hold true for such structures under the extra assumption that the domains of definition and the psudo-definable spaces are definably compact.
Here, a pseudo-definable space is a topological spaces having finite definable atlases.
We also demonstrate Michael's selection theorem for definable set-valued functions with definably compact domains of definition. 
\end{abstract}



\maketitle

\section{Introduction}\label{sec:intro}
Locally o-minimal structures are studied in a series of papers \cite{TV, KTTT, F, Fuji, Fuji3, Fuji4}. 
We first recall the definitions of local o-minimality and definably completeness.
\begin{definition}[\cite{TV}]
A densely linearly ordered structure without endpoints $\mathcal M=(M,<,\ldots)$ is \textit{locally o-minimal} if, for every definable subset $X$ of $M$ and for every point $a\in M$, there exists an open interval $I$ containing the point $a$ such that $X \cap I$ is  a finite union of points and open intervals.
\end{definition}

\begin{definition}[\cite{M}]
An expansion of a densely linearly ordered set without endpoints $\mathcal M=(M,<,\ldots)$ is \textit{definably complete} if any definable subset $X$ of $M$ has the supremum and  infimum in $M \cup \{\pm \infty\}$.
\end{definition}
In \cite{Fuji4}, the first author demonstrated that definably complete locally o-minimal structures satisfying the following property (a) have tame dimension theory and decomposition theorem into good-shaped definable subsets called quasi-special submanifolds.
\begin{enumerate}
\item[(a)] The image of a definable discrete set under a coordinate projection is again discrete.
\end{enumerate}
The first contribution of this paper is that an arbitrary definably complete locally o-minimal structure always enjoys the property (a).
Note that it is not necessarily true for a locally o-minimal structure which is not definably complete.
A counter example is found in \cite[Example 12]{KTTT}.
In Section \ref{sec:property_a}, we demonstrate the property (a)  and recall its consequences given in \cite{Fuji4}.

Using this fact, in the latter part of this paper, we investigate definably complete locally o-minimal expansions of ordered groups when the restriction of multiplication to an arbitrary  bounded open box is definable.
We first explain the reason why we investigate them.

An o-minimal structure has tame topology \cite{vdD} and it enables us to deduce many geometric assertions for definable sets.
When, in addition, it is an expansion of an ordered field, we can get more beautiful assertions such as definable triangulation theorem and definable trivialization theorem.
Our next targets are definably complete locally o-minimal expansions of ordered fields if we employ the same strategy as the study of o-minimal structures.
However, we do not treat them for the following two reasons:

Firstly, studies in this direction have been already done. 
Fornasiero got fruitful results on definably complete local o-minimal expansions of ordered fields in \cite{F}.
Structures called models of DCTC, which are more general than definably complete local o-minimal expansions of ordered fields, are already known to possess the property (a) by \cite[Theorem 4.1]{S}.

The second reason is more crucial.
Many natural definably complete locally o-minimal expansions of ordered fields are inevitably o-minimal.
The most important example is a locally o-minimal expansion of the ordered set of reals $\mathbb R$.
It is trivially o-minimal when it is an expansion of the ordered field of reals.
The first author investigated definably complete locally o-minimal structures admitting local definable cell decomposition in \cite{Fuji}.
They are also o-minimal by \cite[Proposition 2.1]{Fuji} if they are expansions of ordered fields. 
The assumption that the structure is an expansion of an ordered field seems to be too restrictive in the study of definably complete locally o-minimal structures.
 
Therefore, we abandon full definability of multiplication. 
Instead, we assume that the multiplication is partially definable.
This is the reason why we employ the following definition:

\begin{definition}
An expansion $\mathcal M=(M,<,0,+,\ldots)$ of a densely linearly ordered group has \textit{definable bounded multiplication compatible to $+$} if there exist an element $1 \in M$ and a map $\cdot:M \times M \rightarrow M$ such that
\begin{enumerate}
\item[(1)] the tuple $(M,<,0,1,+,\cdot)$ is an ordered field;
\item[(2)] for any bounded open interval $I$, the restriction $\cdot|_{I \times I}$ of the product $\cdot$ to $I \times I$ is definable in $\mathcal M$.
\end{enumerate}
We simply say that $\mathcal M$ has \textit{definable bounded multiplication} when the addition in consideration is clear from the context. 
\end{definition}

The above definition may seem to be artificial.
However, we naturally encounter a definably complete locally o-minimal expansion of an ordered group having bounded multiplication.
We construct an example in Section \ref{sec:basic}.
The basic properties of structures having definable bounded multiplication are also discussed in the same section.

We are now ready to introduce the contexts of Section \ref{sec:functional} and Section \ref{sec:definable_space}.
We conjecture that many assertions which hold true for o-minimal expansions of ordered fields also hold true for definably complete locally o-minimal expansions of ordered groups having definable bounded multiplication under the extra assumption that the relevant spaces are definably compact.
Note that a subset of $M^n$ which is definable in a definably complete expansion of a dense linear order without endpoints is definably compact if and only it it is closed and bounded, where $M$ is the universe of the given structure.
Section \ref{sec:functional} and Section \ref{sec:definable_space} collect affirmative answers to the above conjecture.

{\L}ojasiewicz's inequality and Tietze extension theorem for o-minimal expansions ordered field are familiar to real algebraic geometers and researchers of o-minimal structures  \cite{vdD, vdDM}.
In Section \ref{sec:functional}, we demonstrate that they hold true for definably complete expansions of ordered groups having definable bounded multiplication when the domains of definitions are definably compact.
We also demonstrate Michael's selection theorem in the same situation in the same section.
The strategy for its proof is similar to that for the proof of of \cite{Th}.

We prove that a definably compact pseudo-definable space is definably imbedded into a Euclidean space in Section \ref{sec:definable_space}.
We first define psudo-definable spaces similarly to van den Dries's definable  space \cite[Chapter 10]{vdD} prior to the proof.

In the last of this section, we summarize the notations and terms used in this paper.
The term `definable' means `definable in the given structure with parameters' in this paper.
For a linearly ordered structure $\mathcal M=(M,<,\ldots)$, an open interval is a definable set of the form $\{x \in R\;|\; a < x < b\}$ for some $a,b \in M$.
It is denoted by $(a,b)$ in this paper.
We define a closed interval similarly. 
It is denoted by $[a,b]$.
An open box in $M^n$ is the direct product of $n$ open intervals.
The notation $M_{>r}$ denotes the set $\{x \in M\;|\;x>r\}$ for any $r \in M$.
We set $|x|:=\max_{1 \leq i \leq n}|x_i|$ for any vector $x = (x_1, \ldots, x_n) \in M^n$.
The function $|x-y|$ defines a distance in $M^n$ when $\mathcal M$ is an expansion of an ordered group.
Let $A$ be a subset of a topological space.
The notations $\myint(A)$, $\overline{A}$ and $\partial A$ denote the interior, the closure and the frontier of the set $A$, respectively.

\section{Property ($\mathrm{a}$) and its consequences}\label{sec:property_a}
The purpose of this section is to demonstrate that the property (a) holds true for definably complete locally o-minimal structure and to introduce its consequences.
We first begin with the proof of the property (a).

\subsection{Proof of property (a)}

Let $\mathcal{M}=(M,<,\ldots)$ be a definably complete locally o-minimal structure throughout this subsection.
We get the following lemma.

\begin{lemma}\label{lem:conti_part}
Let $f:I \to M$ be a strictly monotone definable function on an interval $I$.
There exists a subinterval $J \subset I$ such that the restriction of $f$ to $J$ is continuous.
\end{lemma}

\begin{proof}
We only show the assertion when $f$ is strictly increasing on $I$.
The proof is similar when $f$ is strictly decreasing.
We may assume that the image $f(I)$ is bounded.
In fact, take $x<y$ in $I$. The restriction of $f$ to the interval $(x,y)$ has a bounded image contained in $(f(x),f(y))$ because $f$ is strictly increasing. 
Suppose that $f(I)$ is discrete.
Since $f(I)$ is a bounded discrete closed set by \cite[Lemma 2.3]{Fuji4}, $f(I)$ has the maximal element $h=f(x)$ where $x \in I$ by the definable completeness of $\mathcal M$.
Take $y \in I$ with $x<y$.
Since $f$ is strictly increasing on $I$, we have $h=f(x)<f(y) \in f(I)$.
It is a contradiction.
So $f(I)$ contains an interval $(u,v)$ by \cite[Lemma 2.3]{Fuji4}.
Take $c \in f^{-1}(u)$ and $d \in f^{-1}(v)$.
Since the restriction $f|_{(c,d)}:(c,d) \to (u,v)$ of $f$ to the open interval $(c,d)$ is bijective, it is continuous.
\end{proof}

We recall the definition of local monotonicity.

\begin{definition}[Local monotonicity]
A function $f$ defined on an open interval $I$ is \textit{locally constant} if, for any $x \in I$, there exists an open interval $J$ such that $x \in J \subset I$ and the restriction $f|_J$ of $f$ to $J$ is constant.

A function $f$ defined on an open interval $I$ is \textit{locally strictly increasing} if, for any $x \in I$, there exists an open interval $J$ such that $x \in J \subset I$ and $f$ is strictly increasing on the interval $J$.
We define a \textit{locally strictly decreasing} function similarly. 
A \textit{locally strictly monotone} function is a locally strictly increasing function or a locally strictly decreasing function.
A \textit{locally monotone} function is locally strictly monotone or locally constant.
\end{definition}

\begin{theorem}[Strong monotonicity]\label{thm:monotonicity}
Consider a definably complete locally o-minimal structure $\mathcal{M}=(M,<,\ldots)$.
Let  $I$ be an interval and $f:I \to M$ be a definable function.
Then, there exists a mutually disjoint definable partition $I=X_d \cup X_c \cup X_+ \cup X_-$ satisfying the following conditions:
\begin{enumerate}
\item[(1)] the definable set $X_d$ is discrete and closed;
\item[(2)] the definable set $X_c$ is open and $f$ is locally constant on $X_c$;
\item[(3)] the definable set $X_+$ is open and $f$ is locally strictly increasing and continuous on $X_+$;
\item[(4)] the definable set $X_-$ is open and $f$ is locally strictly decreasing and continuous on $X_-$.
\end{enumerate}
\end{theorem}

\begin{proof}
We first define the definable sets $X_c, X_+, X_-$ by
\begin{align*}
X_c=\{x \in M \mid\ &f\text{ is constant on some subinterval of }I\text{ containing }x\},\\
X_+=\{x \in M \mid\ &f\text{ is strictly increasing and continuous }\\
&\text{ on some subinterval of }I\text{ containing }x\} \text{ and }\\
X_-=\{x \in M \mid\ &f\text{ is strictly decreasing and continuous }\\
&\text{ on some subinterval of }I\text{ containing }x\}.
\end{align*}
Set $X_d=I \setminus (X_c \cup X_+ \cup X_-)$.
It is obvious that $X_c,X_+$ and $X_-$ are open and definable.
The remaining task is to show that $X_d$ is discrete and closed.
Suppose not, then $X_d$ contains an interval $J$ by \cite[Lemma 2.3]{Fuji4}.
We may assume that $J$ is bounded without loss of generality.

Fix $a \in J$.
By the local o-miniamlity of $\mathcal{M}$, there exists an open interval $(a,r)$ such that one of followings holds: $f>f(a)$ on $(a,r)$, $f=f(a)$ on $(a,r)$ or $f<f(a)$ on $(a,r)$.
By the definition of $X_d$, we have $f>f(a)$ on $(a,r)$ or $f<f(a)$ on $(a,r)$.
Similar situation holds on the left side of the point $a$.
We define the formula $\Phi_{-,+}(v)$ by 
\begin{eqnarray*}
\exists v_1 \exists v_2 [v_1<v<v_2&\land&\forall w_1 (v_1<w_1<v \to f(w_1)<f(v))\\
&\land&\forall w_2 (v<w_2<v_2 \to f(w_2)>f(v))].
\end{eqnarray*}
The formulas $\Phi_{+,-}(v), \Phi_{+,+}(v), \Phi_{-,-}(v)$ are defined similarly.
We have shown that each $a \in J$ satisfies exactly one of $\Phi_{-,+}(v), \Phi_{+,-}(v), \Phi_{+,+}(v)$ and $\Phi_{-,-}(v)$.
By local o-minimality of $\mathcal{M}$, on some subinterval $J'=(\alpha,\beta)$, one of them holds constantly.

We consider the cases in which $\Phi_{-,+}(x), \Phi_{+,-}(x), \Phi_{+,+}(x)$ and $\Phi_{-,-}(x)$ hold true for all $x \in J'$, separately.
We first consider the following case.
\medskip

\textbf{Case 1.} When we have $\mathcal M \models \Phi_{-,+}(x)$ for all $x \in J'$.
\medskip

For each $x \in J'$, we define $s(x)$ by
\[
s(x)=\sup \{s \in (x,\beta) \mid f(x)<f \mbox{ on } (x,s]\}.
\]
If $s(x)<\beta$, there exists $y \in (s(x),\beta)$ such that $f(x)<f(s(x))<f$ on $(s(x),y)$ because $\mathcal M \models \Phi_{-,+}(s(x))$.
It is a contradiction to the definition of $s(x)$.
So the equality $s(x)=\beta$ holds.
It means that $f$ is strictly increasing on $J'$.
By Lemma \ref{lem:conti_part}, the definable function $f$ is continuous on some subinterval of $J'$.
It contradicts the definition of $X_d$.

When $\mathcal M \models \Phi_{+,-}(x)$ for all $x \in J'$, we can similarly show that $f$ is continuous and strictly decreasing on some subinterval of $J'$.
It is also a contradiction.

The next target is the following case:
\medskip

\textbf{Case 2.} When we have $\mathcal M \models \Phi_{+,+}(x)$ for all $x \in J'$.
\medskip

We define the definable set $B \subset J'$ by
\[
B=\{x \in J' \mid \forall y \in J' (x<y \to f(x)<f(y))\}.
\]
Now the set $B$ must be discrete, closed and bounded.
Otherwise, it contains an interval by \cite[Lemma 2.3]{Fuji4} and $f$ is strictly increasing and continuous on some subinterval by Lemma \ref{lem:conti_part}.
It is a contradiction to the definition of $X_d$.
Since $B$ is discrete, closed and bounded, there exists $\gamma=\max B$.
We set $J_0=(\gamma,\beta) \subset  J'\setminus B$.

Fix $c \in J_0$.
Now we show that there exists $c_0 \in J_0$ such that $f<f(c)$ on $(c_0,\beta)$.
Suppose not and we set $B'=\{y \in J_0 \mid f(c)<f(y)\}$.
By the assumption and local o-minimality of $\mathcal{M}$, there exists $z_0 \in J_0$ such that $(z_0,\beta) \subset B'$.
Set $d=\inf\{z \in [c,\beta) \mid f(c) \leq f \mbox{ on } (z,\beta)\}$.
\begin{itemize}
  \item If $f(d) \geq f(c)$, we have $f(c) < f$ on $(e,\beta)$ for some $e \in (c,d)$ because $\mathcal M \models \Phi_{+,+}(d)$. It is a contradiction to the definition of $d$.
  \item We consider the case in which $f(d) < f(c)$.
 By the definition of $\gamma$, we have $d \notin B$.
 It implies that $f(d) \geq f(e')$ for some $e' \in (d,\beta)$.
 We get $f(e')<f(c)$.
 On the other hand, we obtain $f(c) \leq f(e')$ by the definition of $d$.
 It is a contradiction.
\end{itemize}
We have demonstrated the existence of $c_0$.
So we can define $$y(c)=\inf\{z \in [c,\beta) \mid f<f(c) \text{ on } (z,\beta)\}.$$
Suppose that $c=y(c)$, then there exists $c_1 \in (c,\beta)$ such that $f(c)<f$ on $(c,c_1)$ because $\mathcal M \models \Phi_{+,+}(c)$.
Take $z \in (c,c_1)$, then we have $f(c)<f(z)$.
But, by definition of $y(c)$, we have $f(z)<f(c)$. 
A contradiction.
The inequality $c<y(c)$ holds.

We consider the formula $\Psi_{+,-}(v)$ given by 
\[
\exists v_1 \exists v_2 [\gamma<v_1<v<v_2<\beta \land
\forall z_1 \forall z_2 (v_1<z_1<v<z_2<v_2 \to f(z_1)>f(z_2))].
\]

Now we check that the sentence $\Psi_{+,-}(y(c))$ holds in $\mathcal M$.
Since we have $\mathcal M \models \Phi_{+,+}(y(c))$ by the assumption, we get $f(y_1)>f(y(c))$ for some $y_1 \in (y(c),\beta)$.
Then, there exists $z \in (y(c),y_1)$ such that $f<f(c)$ on $(z,\beta)$ by the definition of $y(c)$.
Since we have $z<y_1$, we obtain $f(y_1)<f(c)$.
They imply the inequality 
\begin{equation}
f(y(c))<f(c). \label{eq:komine1}
\end{equation}
Since $\mathcal{M}$ is locally o-minimal, there exists $v_1 \in (c,y(c))$ such that either $f<f(c)$ or $f \geq f(c)$ on $(v_1,y(c))$.
Suppose that $f<f(c)$ on $(v_1,y(c))$.
The inequality (\ref{eq:komine1}) and the definition of $y(c)$ yield that $f<f(c)$ on $(v_1,\beta)$.
It is a contradiction to the definition of $y(c)$.
So, the inequality $f \geq f(c)$ holds on the interval $(v_1,y(c))$.
On the other hand, for all $z_2 \in (y(c),\beta)$, we get $f(z_2)<f(c)$ by the definition of $y(c)$.
Summarizing the above results, the inequalities $f(z_1) \geq f(c)>f(z_2)$ hold for all $z_1 \in (v_1,y(c))$ and all $z_2 \in (y(c),\beta)$.
It means that $\mathcal M \models \Psi_{+,-}(y(c))$.

So, for all $c \in (\gamma,\beta)$, there exists $y(c) \in (c,\beta)$ such that $\Psi_{+,-}(y(c))$ holds.
Consider the definable set $T=\{x \in (\gamma,\beta)\;|\; \mathcal M \models \Psi_{+,-}(x)\}$.
Since $\mathcal{M}$ is locally o-minimal, there exists  $\delta \in (\gamma,\beta)$ such that the open interval $(\delta,\beta)$ is contained in $T$ or has an empty intersection with $T$.
However, the existence of $y(c)$ implies that $(\delta,\beta)$ is contained in $T$.
We define the formula $\Psi_{-,+}$ similarly to $\Psi_{+,-}$.
We can construct an open interval each of whose elements satisfies both the formulas $\Psi_{+,-}$ and $\Psi_{-,+}$.
It is a contradiction.

We can lead to a contradiction similarly when $\mathcal M \models \Phi_{-,-}(x)$ for all $x \in J'$.
We omit the proof.
We have finally demonstrated that $X_d$ is discrete and closed.
\end{proof}
In \cite{Tsuboi}, A. Tsuboi gave a proof of Theorem 2.2 slightly different from ours.   

We get the following proposition called properties (b) and (c) in \cite{Fuji4}.
\begin{proposition}\label{prop:property_bc}
  The following assertions hold true.
  \begin{enumerate}
  \item[(1)] Let $X_1$ and $X_2$ be definable subsets of $M^m$.
  Set $X=X_1 \cup X_2$.
  Assume that $X$ has a nonempty interior.
  At least one of $X_1$ and $X_2$ has a nonempty interior.
  \item[(2)] Let $A$ be a definable subset of $M^m$ with a nonempty interior and $f:A \rightarrow M^n$ be a definable map.
  There exists a definable open subset $U$ of $M^m$ contained in $A$ such that the restriction of $f$ to $U$ is continuous.
\end{enumerate}
\end{proposition}

\begin{proof}
They follow from Theorem \ref{thm:monotonicity} by \cite[Theorem 2.11.(iii)]{Fuji4}.
\end{proof}

The following is the main theorem of this section, which claims that the property (a) holds true.
\begin{theorem}\label{thm:property_a}
The image of a discrete set definable in a definably complete locally o-minimal structure under a coordinate projection is again discrete.
\end{theorem}

\begin{proof}
Let $D$ be a definable discrete subset of $M^n$ and $\pi:M^n \rightarrow M^d$ be a coordinate projection.
We demonstrate that the image $\pi(D)$ is discrete.

We first reduce to the case in which $d=1$.
Let $\rho_k:M^n \to M$ be the projection on the $k$-th coordinate for $1 \leq k \leq d$.
The image $\rho_k(D)$ is discrete by the assumption.
We have $\pi(D) \subset \prod_{k=1}^d \rho_k(D)$. 
Therefore, $\pi(D)$ is discrete because $\prod_{k=1}^d \rho_k(D)$ is discrete.

We may assume that $d=1$ and $\pi$ is the coordinate projection onto the last coordinate.
We prove the theorem by induction on $n$ when $d=1$.
There is nothing to prove when $n=1$.
Before we consider the case in which $n>1$, we demonstrate the following claim using  the induction hypothesis.
\medskip

\textbf{Claim.} 
Let $l$ and $m$ be nonnegative integers with $l<n$.
Let $X \subset M^l \times M^m$ be a definable subset and $\Pi:M^l \times M^m \rightarrow M^m$ be the coordinate projection onto the last $m$ coordinates.
Assume that the fiber $X_x=\{y \in M^l\;|\; (y,x) \in X\}$ is discrete for any $x \in \Pi(X)$.
There exists a definable map $f:\Pi(X) \rightarrow X$ such that the composition $\Pi \circ f$ is the identity map on $\Pi(X)$.
\medskip

We demonstrate the claim by induction on $l$.
The claim is trivial when $l=0$.
We next consider the case $l=1$.
Fix $c \in M$.
We define the definable function $g:\Pi(X) \to M$ by
\begin{align*}
g(x) = \left\{
\begin{array}{ll}
\inf X_x \cap (c,+\infty) & \text{ if }X_x \cap (c,+\infty) \neq \emptyset \text{ and }\\
\sup X_x \cap (-\infty,c) & \text{ otherwise.}
\end{array}
\right.
\end{align*}
We consider the definable map $f:\Pi(X) \rightarrow M^{1+m}$ given by $f(x)=(g(x),x)$.
Since $X_x$ is discrete and closed by \cite[Lemma 2.3]{Fuji4}, we have $f(x) \in X$ for all $x \in \Pi(X)$.
We have demonstrated when $l=1$.

We consider the case $l>1$.
Let $\Pi_1:M^l \times M^m \rightarrow M^{m+1}$ and $\Pi_2:M^{m+1} \rightarrow M^m$ be the projections onto the last $m+1$ coordinate and onto the last $m$ coordinates, respectively.
We have $\Pi=\Pi_2 \circ \Pi_1$.
It is obvious that the fibers $\{x \in M^{l-1}\;|\;(x,y,z) \in X\}$ with respect to the projection $\Pi_1$ are discrete for all $y \in M$ and $z \in M^m$ with $(y,z) \in \Pi_1(X)$.
There is a definable map $h_1:\Pi_1(X) \rightarrow X$ such that the composition $\Pi_1 \circ h_1$ is the identity map by the induction hypothesis on $l$.
On the other hand, by the induction hypothesis on $n$, the image of a definable discrete subset of $M^l$ under a coordinate projection is discrete.
It implies that the fibers $\{x \in M\;|\;(x,y) \in \Pi_1(X)\}$ with respect to the projection $\Pi_2$ are discrete for all $y \in \Pi(X)$.
The claim for $l=1$ provides a definable map $h_2: \Pi(X) \rightarrow \Pi_1(X)$ such that the composition $\Pi_2 \circ h_2$ is the identity map.
The composition $f=h_1 \circ h_2$ is the desired map.
We have demonstrated the claim.

We return to the proof of the theorem.
There exists a definable map $f:\pi(D) \rightarrow D$ by the claim.
If $\pi(D)$ is not discrete, $\pi(D)$ contains an open interval $I$ by \cite[Lemma 2.3]{Fuji4}.
By Proposition \ref{prop:property_bc}, $f$ is continuous on some subinterval of $I$.
It is a contradiction to the assumption that $D$ is discrete.
We have demonstrated that $\pi(D)$ is discrete.
\end{proof}

We finally get the following proposition called the property (d) in \cite{Fuji4}.
\begin{proposition}\label{prop:property_a}
Let $X$ be a definable subset of $M^n$ and $\pi: M^n \rightarrow M^d$ be a coordinate projection such that the the fibers $X \cap \pi^{-1}(x)$ are discrete for all $x \in \pi(X)$.
Then, there exists a definable map $\tau:\pi(X) \rightarrow X$ such that $\pi(\tau(x))=x$ for all $x \in \pi(X)$.
\end{proposition}

\begin{proof}
It follows from Theorem \ref{thm:property_a} by \cite[Theorem 2.11.(ii)]{Fuji4}.
\end{proof}

\subsection{Consequences of property (a)}

In \cite{Fuji4}, the first author demonstrated several tame topological properties of definable sets and tame dimension theory assuming that the property (a) in Section \ref{sec:intro} holds true.
They hold true without assuming the property (a) thanks to Theorem \ref{thm:property_a}.
This subsection summarizes them for the readers' convenience.

We first define the dimension of a definable set as follows:
\begin{definition}[Dimension]\label{def:dim}
Consider an expansion of a densely linearly order without endpoints $\mathcal M=(M,<,\ldots)$.
Let $X$ be a nonempty definable subset of $M^n$.
The dimension of $X$ is the maximal nonnegative integer $d$ such that $\pi(X)$ has a nonempty interior for some coordinate projection $\pi:M^n \rightarrow M^d$.
We set $\dim(X)=-\infty$ when $X$ is an empty set.
\end{definition}

The following proposition summarizes the results on dimension in \cite{Fuji4}.
\begin{proposition}\label{prop:dim}
Consider a definably complete locally o-minimal structure $\mathcal M=(M,<,\ldots)$.
The following assertions hold true.
\begin{enumerate}
\item[(1)] A definable set is of dimension zero if and only if it is discrete.
When it is of dimension zero, it is also closed.
\item[(2)] Let $X \subset Y$ be definable sets.
Then, the inequality $\dim(X) \leq \dim(Y)$ holds true.
\item[(3)] Let $\sigma$ be a permutation of the set $\{1,\ldots,n\}$.
The definable map $\overline{\sigma}:M^n \rightarrow M^n$ is defined by $\overline{\sigma}(x_1, \ldots, x_n) = (x_{\sigma(1)},\ldots, x_{\sigma(n)})$.
Then, we have $\dim(X)=\dim(\overline{\sigma}(X))$ for any definable subset $X$ of $M^n$.
\item[(4)] Let $X$ and $Y$ be definable sets.
We have $\dim(X \times Y) = \dim(X)+\dim(Y)$.
\item[(5)] Let $X$ and $Y$ be definable subsets of $M^n$.
We have 
\begin{align*}
\dim(X \cup Y)=\max\{\dim(X),\dim(Y)\}\text{.}
\end{align*}
\item[(6)] Let $f:X \rightarrow M^n$ be a definable map. 
We have $\dim(f(X)) \leq \dim X$.
\item[(7)] Let $f:X \rightarrow M^n$ be a definable map. 
The notation $\mathcal D(f)$ denotes the set of points at which the map $f$ is discontinuous. 
The inequality $\dim(\mathcal D(f)) < \dim X$ holds true.
\item[(8)] Let $X$ be a definable set.
The notation $\partial X$ denotes the frontier of $X$ defined by $\partial X = \overline{X} \setminus X$.
We have $\dim (\partial X) < \dim X$.
\item[(9)] A definable set $X$ is of dimension $d$ if and only if the nonnegative integer $d$ is the maximum of nonnegative integers $e$ such that there exist an open box $B$ in $M^e$ and 
a definable injective continuous map $\varphi:B \rightarrow X$ homeomorphic onto its image. 
\item[(10)] Let $X$ be a definable subset of $M^n$.
There exists a point $x \in X$ such that we have $\dim(X \cap B)=\dim(X)$ for any open box $B$ containing the point $x$.
\item[(11)] Let $\varphi:X \rightarrow Y$ be a definable surjective map whose fibers are equi-dimensional; that is, the dimensions of the fibers $\varphi^{-1}(y)$ are constant.
We have $\dim X = \dim Y + \dim \varphi^{-1}(y)$ for all $y \in Y$.  
\end{enumerate}
\end{proposition}
\begin{proof}
See \cite[Proposition 3.2, Theorem 3.8, Theorem 3.11, Corollary 3.12, Theorem 3.14]{Fuji4}
\end{proof}

\begin{corollary}\label{cor:key}
Let $\mathcal M=(M,<,\ldots)$ be a definably complete locally o-minimal structure.
Let $B$ and $C$ be definable open subsets of $M^m$ and $M^n$, respectively.
If there exists a definable injective map from $B$ to $C$, we have $m \leq n$.
\end{corollary}
\begin{proof}
Immediate from Proposition \ref{prop:dim}(1), (2) and (11).
\end{proof}

We finally recall the notion of quasi-special submanifolds and the decomposition theorem into quasi-special submanifolds.

\begin{definition}\label{def:quasi-special}
Consider an expansion of a densely linearly order without endpoints $\mathcal M=(M,<,\ldots)$.
Let $X$ be a definable subset of $M^n$ and $\pi:M^n \rightarrow M^d$ be a coordinate projection.
A point $x \in X$ is \textit{($X,\pi$)-normal} if there exists an open box $B$ in $M^n$ containing the point $x$ such that $B \cap X$ is the graph of a continuous map defined on $\pi(B)$ after permuting the coordinates so that $\pi$ is the projection onto the first $d$ coordinates.

A definable subset is a \textit{$\pi$-quasi-special submanifold} or simply a \textit{quasi-special submanifold} if, $\pi(X)$ is a definable open set and, for every point $x \in \pi(X)$, there exists an open box $U$ in $M^d$ containing the point $x$ satisfying the following condition:
For any $y \in X \cap \pi^{-1}(x)$, there exist an open box $V$ in $M^n$ and a definable continuous map $\tau:U \rightarrow M^n$ such that $\pi(V)=U$, $\tau(U)=X \cap V$ and the composition $\pi \circ \tau$ is the identity map on $U$.

Let $\{X_i\}_{i=1}^m$ be a finite family of definable subsets of $M^n$.
A \textit{decomposition of $M^n$ into quasi-special submanifolds partitioning $\{X_i\}_{i=1}^m$} is a finite family of quasi-special submanifolds $\{C_i\}_{i=1}^N$ such that $\bigcup_{i=1}^NC_i =M^n$, $C_i \cap C_j=\emptyset$ when $i \not=j$ and $C_i$ has an empty intersection with $X_j$ or is contained in $X_j$ for any $1 \leq i \leq m$ and $1 \leq j \leq N$.
A decomposition $\{C_i\}_{i=1}^N$ of $M^n$ into quasi-special submanifolds \textit{satisfies the frontier condition} if the closure of any quasi-special submanifold $\overline{C_i}$ is the union of a subfamily of the decomposition.
\end{definition}

\begin{proposition}\label{prop:frontier_condition}
Consider a definably complete locally o-minimal structure $\mathcal M=(M,<,\ldots)$.
Let $\{X_i\}_{i=1}^m$ be a finite family of definable subsets of $M^n$.
There exists a decomposition $\{C_i\}_{i=1}^N$ of $M^n$ into quasi-special submanifolds partitioning $\{X_i\}_{i=1}^m$ and satisfying the frontier condition.
Furthermore, the number $N$ of quasi-special submanifolds is not greater than the number uniquely determined only by $m$ and $n$.
\end{proposition}
\begin{proof}
See \cite[Theorem 4.5]{Fuji4}.
\end{proof}

\section{Basic property of structures having definable bounded multiplication}\label{sec:basic}

We defined an expansion of an ordered group having definable multiplication in Section \ref{sec:intro}.
We can easily construct a definably complete locally o-minimal expansion of an ordered group having definable multiplication using the notion of simple product in \cite[Section 4]{KTTT}.
The following examples is also given in \cite[Example 4.12]{Fuji6}.

\begin{example}\label{ex:mult}
Let $\widetilde{\mathbb R}$ be an arbitrary o-minimal expansion of the ordered group of reals.
We define a locally o-minimal expansion $\mathcal M:=\mathcal M(\widetilde{\mathbb R})=(\mathbb R,<,0,+,\mathbb Z,\ldots)$ of the ordered group of reals.
We define that a subset $X$ of $\mathbb R^n$ is definable in $\mathcal M$ if and only if, there exist finitely many subsets $X_1, \ldots, X_k$ of $[0,1)^n$ definable in $\widetilde{\mathbb R}$ and a map $\tau:\mathbb Z^n \rightarrow \{1, \ldots, k\}$ such that, for any $(m_1,\ldots, m_n) \in \mathbb Z^n$, we have 
\begin{align*}
X \cap \left(\prod_{i=1}^n [m_i,m_i+1)\right) &= (m_1,\ldots, m_n)+X_{\tau(m_1,\ldots, m_n)}\text{,}
\end{align*}
where the right hand of equality is 
\begin{align*}
&\{(x_1,\ldots, x_n) \in \mathbb R^n\;|\; (x_1-m_1,\ldots, x_n-m_n) \in X_{\tau(m_1,\ldots, m_n)}\}\text{.}
\end{align*}
A standard argument shows that $\mathcal M$ is a locally o-minimal structure.
The structure $\mathcal M$ is not an o-minimal structure.
It is also easy to demonstrate that a bounded subset $X$ of $\mathbb R^n$ is definable in $\mathcal M$ if it is definable in $\widetilde{\mathbb R}$.
We omit the details.

Consider the graph of the addition $\Gamma=\{(x,y,z) \in \mathbb R^3\;|\; z=x+y\}$.
Set $X_1=\Gamma \cap [0,1)^3$ and $X_2 :=\{(x,y,z) \in [0,1)^3\;|\;x+y \geq 1,\ z=x+y-1\}$.
We obviously have 
\begin{align*}
\Gamma \cap \prod_{i=1}^3[m_i,m_i+1) = \left\{
\begin{array}{ll}
(m_1,m_2,m_3)+X_1 & \text{ if }m_3=m_1+m_2,\\
(m_1,m_2,m_3)+X_2 & \text{ if }m_3=m_1+m_2+1 \text{ and }\\
\emptyset & \text{ otherwise.}
\end{array}
\right.
\end{align*}
It implies that the graph $\Gamma$ is definable in $\mathcal M$.
The structure $\mathcal M$ is an expansion of an ordered group.

When $\widetilde{\mathbb R}$ is an expansion of an ordered field, the structure $\mathcal M$ has definable bounded multiplication compatible to $+$ because any bounded set definable in $\widetilde{\mathbb R}$ is also definable in $\mathcal M$.
\end{example}

We give several basic properties of a structure having definable bounded multiplication.
We defined a uniformly locally o-minimal structure of the second kind in \cite{Fuji}.
It becomes almost o-minimal defined in \cite{Fuji6} when it has definable bounded multiplication.
\begin{proposition}\label{prop:almost}
A uniformly locally o-minimal expansion of the second kind of an ordered abelian group having definable bounded multiplication is almost o-minimal.
In particular, it is definably complete.
\end{proposition}
\begin{proof}
Let $\mathcal M=(M,<,0,+,\ldots)$ be the structure in consideration.
 We have only to demonstrate that a bounded definable set $X$ of $M$ is a finite union of points and open intervals.
Let $R$ be a positive element such that $X$ is contained in the open interval $(-R,R)$.
Consider the set $$Y=\{(r,rx) \in M^2\;|\; 0<r<1 \text{ and }x \in X\}\text{.}$$
It is definable because we only use multiplication in a bounded area for its definition. 
Since $\mathcal M$ is uniformly locally o-minimal of the second kind, we can take $\delta>0$ and $\varepsilon>0$ such that, for any $0<r<\varepsilon$, the intersection $Y_r \cap (-\delta,\delta)$ is a finite union of points and open intervals, where $Y_r$ is the fiber at $r$ given by $Y_r:=\{x \in M\;|\; (r,x) \in Y\}$.
Take a sufficiently small $r>0$ such that $rR<\delta$ and $r<\varepsilon$.
For such an $r$, we have $Y_r\cap (-\delta,\delta)=rX:=\{rx \in M\;|\; x \in X\}$.
It implies that $X$ is a finite union of points and open intervals.

The almost o-minimal structure is definably complete by \cite[Lemma 4.6]{Fuji6}.
\end{proof}

A definably complete expansion of an ordered abelian group having definable bounded multiplication is always a real closed field.
\begin{proposition}\label{prop:rcf}
Consider a definably complete expansion of an ordered abelian group having definable bounded multiplication $\mathcal M=(M,<,0,+,\ldots)$.
Then, $(M,<,0,1,+,\cdot)$ is a real closed field.
\end{proposition}
\begin{proof}
The restriction of a polynomial function to a bounded interval is definable because $\mathcal M$ has definable bounded multiplication.
Use this fact together with the intermediate value property of definably complete structures \cite{M} and \cite[Lemma 1.2.11]{BCR} which is valid for an ordered field.
We can easily check that the condition \cite[Theorem 1.2.2(ii)]{BCR} is satisfied.
\end{proof}

The following proposition and its corollary are trivial.
\begin{proposition}\label{prop:semialg}
Let $\mathcal M=(M,<,0,+,\ldots)$ be a definably complete expansion of an ordered abelian group having definable bounded multiplication.
Any bounded semialgebraic set is definable.
\end{proposition}
\begin{proof}
Obvious.
\end{proof}

\begin{corollary}
Let $\mathcal M=(M,<,0,+,\ldots)$ be a definably complete expansion of an ordered abelian group having definable bounded multiplication.
The $n$-dimensional projective space $\mathbb P_n(M)$ and the Grassmannian $\mathbb G_{n,k}(M)$ are definable in $\mathcal M$. 
\end{corollary}
\begin{proof}
Because they are closed and bounded algebraic sets by \cite[Theorem 3.4.4, Proposition 3.4.11]{BCR}.
\end{proof}

The following lemma is used in the subsequent sections.

\begin{lemma}\label{lem:basic}
Consider an expansion $\mathcal M=(M,<,0,+,\ldots)$ of a densely linearly ordered group having definable bounded multiplication $\cdot$ compatible to $+$.
Consider a definable set $S$ and bounded definable functions $f,g:S \rightarrow M$ such that $f(x) \neq 0$ for all $x \in S$.
Then, the function $h:S \rightarrow M$ given by $h(x)=g(x)/f(x)$ is definable when $h$ is also bounded. 
\end{lemma}
\begin{proof}
Immediate from the definition of definable bounded multiplication.
\end{proof}

A natural question is whether a derivative of a definable function is again definable.
It is true in some restricted case, but false in general. 

\begin{lemma}\label{lem:c1}
Consider an expansion $\mathcal M=(\mathbb R,<,0,+,\ldots)$ of the ordered group of reals having definable bounded multiplication $\cdot$ compatible to $+$.
Let $I$ be a bounded closed interval and $f:I \rightarrow M$ be a definable $C^1$ function.
Then, its derivative $f':I \rightarrow M$ is also definable.
\end{lemma}
\begin{proof}
Consider the function $h: \{(x,h) \in I \times M\;|\; h \neq 0, x+h \in I\} \rightarrow M$ given by $$h(x)=\dfrac{f(x+h)-f(x)}{h}\text{.}$$
Since $f'$ is continuous and $I$ is compact, $f'$ is bounded.
Lemma \ref{lem:basic} implies that the function $h$ is definable.
Therefore,  the derivative $f'$ is also definable because the closure of a definable set is again definable.
\end{proof}

\begin{example}
In Lemma \ref{lem:c1}, the assumption that $I$ is closed is inevitable.
Let $\widetilde{\mathbb R}=(\mathbb R, <, 0, 1,+ , \cdot, e^x)$ denote the expansion of the ordered fields of reals by the exponentiation. 
It is o-minimal by \cite{W}.
Let $\mathcal M(\widetilde{\mathbb R})$ be the locally o-minimal structure defined in Example \ref{ex:mult}.
The function $f:(0,1) \rightarrow \mathbb R$ given by $f(x)=x \log x$ is definable in $\mathcal M(\widetilde{\mathbb R})$ because it is a bounded function definable in $\widetilde{\mathbb R}$.
Its derivative $f'(x)=\log x+1$ is not definable in $\mathcal M(\widetilde{\mathbb R})$ by \cite[Lemma 5.3]{Fuji5}.
\end{example}

\section{Functional results for locally o-minimal structures having definable bounded multiplication}\label{sec:functional}
\subsection{{\L}ojasiewicz's inequality}
{\L}ojasiewicz's inequality and its variants are proved in many model-theoretic expansions of an ordered group. 
We prove that it holds true for definably complete expansion of an ordered group having definable bounded multiplication.
\begin{definition}\label{def:Phi}
We consider an expansion $\mathcal M=(M,<,0,+,\ldots)$ of a densely linearly ordered group having definable bounded multiplication $\cdot$ compatible to $+$.
The notation $\Phi_0(M)$ denotes the set of all odd strictly increasing definable bijection from $M$ onto $M$ fixing the origin.
It is a group with respect to composition.
We fix a subgroup $\Phi \subseteq \Phi_0(M)$ satisfying the following condition:
For any positive $a \in M$ and any definable continuous strictly increasing function $f:[0,a] \rightarrow M$ with $f(0)=0$, there are $\varphi, \psi \in \Phi$ such that $f(t) \leq \varphi(t)$ and $\psi(t) \leq f(t)$ for all sufficiently small $t>0$.

We call that $\mathcal M$ is \textit{polynomially bounded} if, for any positive $a \in M$ and any definable continuous strictly increasing function $f:[0,a] \rightarrow M$ with $f(0)=0$, there exists a positive integer $m \in \mathbb Z$ such that $f(t) \geq t^m$ for all sufficiently small $t>0$.
\end{definition}

\begin{example}
The group $\Phi=\Phi_0(M)$ satisfies the condition in Definition \ref{def:Phi}.
Let $a\in M$ be a positive element and $f:[0,a] \rightarrow M$ be a definable continuous strictly increasing function with $f(0)=0$.
The definable functions $\varphi, \psi \in \Phi$ defined below satisfy the inequalities in the definition.
\begin{align*}
\varphi(t)=\psi(t)=\left\{\begin{array}{ll} t-a+f(a) & \text{ if }t>a,\\ f(t) & \text{ if } 0 \leq t \leq a,\\ -f(-t) & \text{ if } -a \leq t<0 \text{ and }\\ t+a-f(a) & \text{ if }t < -a \text{.}\end{array}\right.
\end{align*}
\end{example}

\begin{example}
We constructed a locally o-minimal expansion $\mathcal M(\widetilde{\mathbb R})$ of the ordered group of reals from an o-minimal expansion $\widetilde{\mathbb R}$ of the ordered group of reals in Example \ref{ex:mult}.
When $\widetilde{\mathbb R}$ is a polynomially bounded expansion of the ordered field of reals, the locally o-minimal structure $\mathcal M(\widetilde{\mathbb R})$ is also polynomially bounded in the sense of Definition \ref{def:Phi}.
See \cite{hard} for the definition of a polynomially bounded expansion of the ordered field of reals.
\end{example}

\begin{lemma}\label{lem:polynomially_bounded}
Consider a definably complete expansion $\mathcal M=(M,<,0,+,\ldots)$ of a densely linearly ordered group having definable bounded multiplication $\cdot$ compatible to $+$.
Assume further that $\mathcal M$ is polynomially bounded.
For any $\varphi \in \Phi$ and any positive $a \in M$, there exist a positive $c \in M$ and a positive odd integer $m$ such that $\varphi(t) \geq ct^m$ for all $0 \leq t \leq a$. 
\end{lemma}
\begin{proof}
By the assumption, there exist $0<b \leq a$ and a positive integer such that $\varphi(t) \geq t^m$ for all $0 \leq t < b$.
We may assume that $b<1$.
We may also assume that $m$ is odd by increasing $m$ if necessary.
Since $\varphi$ is positive, we have $s=\inf\{\varphi(t)\;|\; b \leq t \leq a\}>0$ by \cite[Corollary, p.1786]{M}.
Set $c=\min\{s/a^m,1\}$.
We obviously get the inequality $\varphi(t) \geq ct^m$ for all $0 \leq t \leq a$. 
\end{proof}

Consider a definably complete expansion $\mathcal M=(M,<,0,+,\ldots)$ of a densely linearly ordered group.
For any $x=(x_1,\ldots, x_n), y=(y_1,\ldots, y_n) \in M^n$, we set 
\[
d(x,y)=\max_{1 \leq i \leq n}|x_i-y_i|\text{.}
\]
The function $d$ is a distance function on $M^n$.
Let $A$ be a closed definable subset of $M^n$.
The distance function $d_A:M^n \rightarrow M$ is given by 
\[
d_A(x)=\inf\{d(x,a) \;|\; a \in A\}\text{.}
\]

\begin{proposition}[{\L}ojasiewicz's inequality]\label{prop:Lojasiewicz}
Consider a definably complete locally o-minimal expansion $\mathcal M=(M,<,0,+,\ldots)$ of an ordered group having definable bounded multiplication.
Let $S$ be a definable, closed and bounded subset of $M^n$.
Let $f,g:S \rightarrow M$ be definable continuous functions with $f^{-1}(0) \subseteq g^{-1}(0)$.
Then, we can choose a $\phi \in \Phi$ and a positive $N \in M$ such that $$|\phi(g(x))| \leq N|f(x)|$$ for all $x \in S$.

In addition, when $\mathcal M$ is polynomially bounded, there exists a positive integer $m$ such that $$|g(x)|^m \leq N|f(x)|$$ for all $x \in S$.
\end{proposition}
\begin{proof}
Set $Z=f^{-1}(0)$.
The notation $d_Z$ denotes the distance function to $Z$ in $M^n$.
When $s=\inf\{d_Z(x)\;|\;x \in S \setminus Z\}>0$, the definable set $S \setminus Z$ is closed.
By the max-min principle \cite{M}, the image of a bounded closed definable set under a definable continuous function is  closed and bounded.
We use this fact without mention.
There exists $c>0$ with $c \leq |f(x)|$ for all $x \in S \setminus Z$.
The definable functions $f(x)$ and $g(x)$ are bounded.
The restriction of $g(x)/f(x)$ to $S \setminus Z$ is bounded.
It is also definable by Lemma \ref{lem:basic}.
We can take $N>0$ such that $|g(x)/f(x)| \leq N$.
We get $|g(x)| \leq N|f(x)|$ for all $x \in S \setminus Z$.
The inequality $|g(x)| \leq N|f(x)|$ is clear when $x \in Z$.
We have proved the proposition when $s>0$.
We can prove the inequality in the same manner when $Z$ is an empty set.

We consider the case in which $Z \neq \emptyset$ and $s=0$.
We can take $u>0$ such that $[0,u) \subseteq d_Z(S)$ by local o-minimality.
A definable continuous function on a definable, bounded and closed set is uniformly continuous by \cite[Corollary 4.3, Remark 4.13]{Fuji5}.
We also use this fact. 
Consider the function $\eta:[0,u) \rightarrow M$ defined by 
$$\eta(t)=\inf\{|f(x)|\;|\; x \in S,\ d_Z(x)=t\}\text{.}$$
When $0<t<u$, we have $\eta(t)>0$ by the max-min principle.
By uniform continuity of $|f(x)|$, for any $\varepsilon>0$, we can choose $\delta>0$ such that $|f(x)| < \varepsilon$ for all $d_Z(x)<\delta$.
It means that $\eta$ is continuous at $0$.
We may assume that, by taking a smaller $u$ if necessary, $\eta$ is continuous by the strong local monotonicity property.
Consider the function $\rho:[0,u) \rightarrow M$ defined by 
$$\rho(t)=\sup\{|g(x)|\;|\; x \in S,\ d_Z(x)=t\}\text{.}$$
We may assume that $\rho$ is continuous for the same reason as above.

By Definition \ref{def:Phi}, there exist $\phi_1,\phi_2 \in \Phi$ such that 
\begin{align}
& \phi_1(t) \leq \eta(t) \label{eq:1} \text{ and }\\
&\rho(t) \leq \phi_2(t) \label{eq:2}
\end{align}
for all sufficiently small $t \geq 0$.
We may assume that the inequalities (\ref{eq:1}) and (\ref{eq:2}) hold true for all $0 \leq t <u$ by taking a smaller $u>0$ again.
Set $T=\{x \in S\;|\; d_Z(x) < u\}$ and $\phi=\phi_1 \circ \phi_2^{-1}$.
Fix an arbitrary $x \in T$.
Set $t=d_Z(x)$.
We have $t<u$.
Using the fact $\phi$ is strictly increasing odd continuous function, we get
\begin{align}
|\phi(g(x))| & \leq \sup\{ |\phi(g(y))| \;|\; y \in S,\ d_Z(y)=t\}\label{eq:3}\\
&=\phi(\sup\{ |g(y)| \;|\; y \in S,\ d_Z(y)=t\})\nonumber\\
&=\phi(\rho(t)) \leq \phi(\phi_2(t))=\phi_1(t) \leq \eta(t)\nonumber\\
& \leq |f(x)|\nonumber
\end{align}

We consider the case in which $x \in S \setminus T$.
The function $|\phi(g(x))/f(x)|$ is bounded on $S \setminus T$ for the same reason as the case in which $s=0$.
We can take $N>1$ such that $|\phi(g(x))| \leq N|f(x)|$ for all $x \in S \setminus T$.
This inequality and the inequality (\ref{eq:3}) deduces the inequality $|\phi(g(x))| \leq N|f(x)|$.

We finally consider the case in which $\mathcal M$ is polynomially bounded.
By the max-min principle, there exists $a>0$ such that $|g(x)|<a$ for all $x \in S$.
We can take a positive odd integer $m$ and a positive constant $c \in M$ such that $\phi(t) \geq ct^m$ for all $0 \leq t \leq a$ by Lemma \ref{lem:polynomially_bounded}. 
We therefore have $|\phi(g(x))| \geq c|g(x)|^m$.
We finally get $|g(x)|^m \leq (N/c)|f(x)|$.
\end{proof}

\begin{proposition}\label{prop:zero}
Let $\mathcal M$, $S$, $f$ and $g$ be the same as in Proposition \ref{prop:Lojasiewicz}.
There exist $\phi \in \Phi$ and a definable continuous function $h:S \rightarrow M$ such that $\phi(g(x))=h(x)f(x)$ for all $x \in S$.

In addition, when $\mathcal M$ is polynomially bounded, there exists a positive integer $m$ such that $(g(x))^m=h(x)f(x)$.
\end{proposition}
\begin{proof}
There exist a $\psi \in \Phi$ and a positive $N \in M$ such that 
\begin{equation}
|\psi(g(x))| \leq N|f(x)|\label{eq:11}
\end{equation}
for all $x \in S$ by Proposition \ref{prop:Lojasiewicz}.
Since $g$ is bounded, we may assume that $|g(S)|<a$ for some positive $a$.
The restriction of $t^2\psi(t)$ to $[0,a]$ is definable in $\mathcal M$.
Apply Definition \ref{def:Phi} to it.
We can choose a $\phi \in \Phi$ such that $\phi(t) \leq t^2\psi(t)$ for all sufficiently small $t \geq 0$.
We may assume that $\phi(t)=t^m$ for some positive odd integer $m$ by the definition when $\mathcal M$ is polynomially bounded.

Set $Z=f^{-1}(0)$.
Define a function $h:S \rightarrow M$ by
$$
h(x) = \left\{\begin{array}{ll} \phi(g(x))/f(x) & \text{ when } x \not\in Z\\ 0 & \text{ otherwise.}\end{array}\right.
$$
We have only to demonstrate that $h$ is definable and continuous.
By the inequality (\ref{eq:11}), we have 
\begin{equation}
|\phi(g(x))/f(x)| \leq Ng(x)^2 \label{eq:12}
\end{equation}
for all $x \not\in Z$.
In particular, the function $\phi(g(x))/f(x)$ on $S \setminus Z$ is bounded.
Therefore, it is definable by Lemma \ref{lem:basic}.
We have demonstrated that $h$ is definable.
The continuity of $h$ is also obvious from the inequality (\ref{eq:12}) because $g$ is continuous and identically zero on $Z$.
\end{proof}

\subsection{Weak Tietze extension theorem}

Mimicing the proof of \cite[Proposition 2.6.9]{BCR}, we get a weak Tietze extension theorem.
\begin{theorem}[Weak Tietze extension theorem]\label{thm:weak_tietze}
Consider a definably complete locally o-minimal expansion $\mathcal M=(M,<,0,+,\ldots)$ of an ordered group having definable bounded multiplication.
Let $S$ be a definable, closed and bounded subset of $M^n$.
A definable continuous function $f:S \rightarrow M$ has its definable continuous extension $\overline{f}:M^n \rightarrow M$.
\end{theorem}
\begin{proof}
Set $f^+=\frac{1}{2}(f+|f|)$ and $f^-=\frac{1}{2}(|f|-f)$.
We have $f=f^+-f^-$, $f^- \geq 0$ and $f^+ \geq 0$ on $S$.
Hence, we may assume that $f \geq 0$ without loss of generality.

Apply Proposition \ref{prop:Lojasiewicz} to $f(x)-f(y)$ and the restriction of the distance function $d$ to $S \times S$.
There exists a $\phi \in \Phi$ and a positive $N \in M$ such that
$$|\phi(f(x)-f(y))| \leq Nd(x,y)$$
for all $x,y \in S$.
Consider the definable functions $\Delta:M^n \times M^n \rightarrow M$ and $\overline{f}:M^n \rightarrow M$ defined by
\begin{align*}
&\Delta(x,y) = \phi^{-1}(Nd(x,y))+d(x,y) \text{ and }\\
&\overline{f}(x)=\inf \{\Delta(x,y)+f(y)\;|\;y \in S\} \text{.}
\end{align*}
The function $\overline{f}$ is well-defined because $\mathcal M$ is definably complete and $\Delta$ is continuous.
We show that $\overline{f}$ is a definable continuous extension of $f$.

We first demonstrate that $\overline{f}(x)=f(x)$ for all $x \in S$.
Let $y$ be an arbitrary element in $S$.
We get
\begin{align*}
|f(x)-f(y)| &= \phi^{-1} \circ \phi(|f(x)-f(y)|) = \phi^{-1}(|\phi(f(x)-f(y))|)\\
&\leq \phi^{-1}(Nd(x,y)) \leq \Delta(x,y)\text{.}
\end{align*} 
It implies that $f(x) \leq f(y)+\Delta(x,y)$.
We get $\overline{f}(x)=f(x)$ by the definition of $\overline{f}$.

The next task is to prove that $\overline{f}$ is continuous.
Take an arbitrary element $a \in M^n$ and set $b=\overline{f}(a)$.
Fix an arbitrary positive element $\varepsilon \in M$.
We demonstrate that there exists $\delta>0$ such that $b-\varepsilon < \overline{f}(x) < b+ \varepsilon$ whenever $d(x,a)<\delta$.

We obtain $b=\inf\{\Delta(a,y)+f(y)\;|\;y \in S\}$.
By the max-min principle, we get $b=\Delta(a,y')+f(y')$ for some $y' \in S$.
Since the function $\Delta$ is continuous, we can take $\eta>0$ so that $|\Delta(x,y')+f(y')-b|=|\Delta(x,y')-\Delta(a,y')|<\varepsilon$ whenever $d(x,a) < \eta$.
It implies that $\overline{f}(x)<b+\varepsilon$.

We next demonstrate that we can take $\mu>0$ so that $\overline{f}(x) > b- \varepsilon$ whenever $d(x,a)<\mu$.
Once this claim is proved, the positive element $\delta=\min\{\eta,\mu\}$ satisfies the required condition and the assertion that $\overline{f}$ is continuous is demonstrated.
Assume for the contradiction that, for any sufficiently small $\mu>0$, there exist $x \in M^n$, $y \in S$ such that $d(x,a)<\mu$ and $\Delta(x,y)+f(y) \leq b-\varepsilon$.
Consider the definable closed set
$$K=\{(x,y) \in M^n \times S\;|\; \Delta(x,y)+f(y) \leq b - \varepsilon\}\text{.}$$
It is contained in the bounded set $\{(x,y)\in M^n \times C\;|\; d(x,y) \leq b -\varepsilon\}$ because of the inequalities $f \geq 0$ and $d(x,y) \leq \Delta(x,y)$.
In particular, it is bounded.
Let $\pi:M^n \times M^n \rightarrow M^n$ be the coordinate projection onto the first $n$ coordinates.
The projection image $\pi(K)$ is closed, bounded and definable by \cite[Lemma 1.7]{M}.
On the other hand, the assumption implies that the point $a$ is contained in the closure of $\pi(K)$.
Therefore, we get $a \in \pi(K)$.
It means that $\Delta(a,y)+f(y) \leq b-\varepsilon$ for some $y \in C$, which contradicts to the assumption that $\overline{f}(a)=b$.
We have finished the proof.
\end{proof}

\begin{remark}
An expansion $\mathcal M=(M,<,\ldots)$ of a dense linear order without endpoints enjoys  \textit{definable Tietze extension property} if, for any positive integer $n$, any definable closed subset $A$ of $M^n$ and any continuous definable function $f:A \rightarrow M$, there exists a definable continuous extension $F:M^n \rightarrow M$ of $f$.

An archimedean definably complete uniformly locally o-minimal expansion $\mathcal M$ of the second kind of an ordered group is o-minimal when $\mathcal M$ enjoys definable Tietze extension property \cite[Theorem 5.5]{Fuji5}.
We constructed a definably complete expansion of the ordered group of reals having definable bounded multiplication compatible to the addition in Example \ref{ex:mult}.
Theorem \ref{thm:weak_tietze} is valid for it, but it does not enjoy definable Tietze extension property.
\end{remark}

\subsection{Michael's selection theorem}

Michael's selection theorem is popular in variational analysis \cite{AF, RW}. 
A similar argument to \cite{Th} yields a weak Michael's selection theorem for definable set-valued maps.
We first recall several definitions.

\begin{definition}
For sets $X$ and $Y$, we denote a map $T$ from $X$ to the power set of $Y$ by $T: X \rightrightarrows Y$ and call it a \textit{set-valued map}.
When $X$ and $Y$ are topological spaces, a \textit{continuous selection} of a set-valued map $T:X \rightrightarrows Y$ is a continuous function $f:X \rightarrow Y$ such that $f(x) \in T(x)$ for all $x \in X$.  

Consider an expansion $\mathcal M=(M,<,\ldots)$ of a dense linear order without endpoints.
Let $E$ be a definable subset of $M^n$.
We consider a \text{definable set-valued map} $T:E \rightrightarrows M^m$; that is, the \textit{graph} $\Gamma(T):=\bigcup_{x \in E}\{x\} \times T(x) \subseteq E \times M^m$ is definable.
We define a \textit{definable continuous selection} of $T$, similarly.

A set valued map $T:E \rightrightarrows M^m$ is \textit{lower semi-continuous} if, for any $x_0 \in E$, $y_0 \in T(x_0)$ and a neighborhood $V$ of $y_0$, there exists a neighborhood $U$ of $x_0$ such that $T(x) \cap V \neq \emptyset$ for all $x \in U$.
A lower semi-continuous set-valued map $T:E \rightrightarrows M^m$ is \textit{continuous} if its graph $\Gamma(T)$ is closed in $E \times M^m$.
\end{definition}

We prove the weak Michael's selection theorem.
\begin{theorem}[Weak Michael's selection theorem]\label{thm:michael}
Consider a definably complete locally o-minimal expansion $\mathcal M=(M,<,0,+,\ldots)$ of an ordered group having definable bounded multiplication.
Let $E$ be a closed, bounded and definable subset of $M^n$ and $T:E \rightrightarrows M^m$ be a definable lower semi-continuous set-valued map such that $T(x)$ are closed and convex for all $x \in E$.
The set-valued map $T$ has a definable continuous selection.
\end{theorem}

We fix a structure $\mathcal M=(M,<,0,+,\ldots)$ satisfying the conditions in Theorem \ref{thm:michael} until its proof is completed.

We prepare several definitions and lemmas for proving Theorem \ref{thm:michael}.
\begin{definition}
The notation $|\!| x|\!|$ denotes the Euclidean norm in $M^m$.
Let $E$ be a subset of $M^n$ and $T:E \rightrightarrows M^m$ be a set-valued map.
Let $f:E \rightarrow M^m$ be a map.
The notation $T-f:E \rightrightarrows M^m$ denotes the set-valued map given by $(T-f)(x)=\{y-f(x)\;|\; y \in T(x)\}$.

We define the \textit{least norm selection of $T$} when $T(x)$ are closed and convex for all $x \in E$.
The unique point $y \in T(x)$ whose Euclidean norm $|\!| y|\!|$ is the smallest in $T(x)$ is denoted by $\mylns_T(x)$.
The existence and uniqueness easily follow from the assumption that $T(x)$ are closed and convex.
It induces a map $\mylns_T:E \rightarrow M^m$.
\end{definition}
 
 The following proof is straightforward.
\begin{lemma}\label{lem:michael_b}
Let $T: E \rightrightarrows M^m$ be a lower semi-continuous set-valued map and $f:E \rightarrow M$ be a continuous map.
Then, the set-valued map $T-f$ is also lower semi-continuous.
Furthermore, $T-f$ is continuous when $T$ is continuous.
\end{lemma}

We also need the following lemma:
\begin{lemma}\label{lem:michael_lns}
Let $T: E \rightrightarrows M^m$ be a continuous set-valued map such that $T(x)$ are closed and convex for all $x \in E$.
Then, the least norm selection $\mylns_T:E \rightarrow M^m$ is a continuous selection of $T$.
\end{lemma}
\begin{proof}
We have $\mylns_T(x) \in T(x)$ because $T(x)$ is closed and convex.
The remaining task is to demonstrate that $\mylns_T$ is continuous. 

Assume for contradiction that $\mylns_T$ is discontinuous at $a \in E$.
Let $G$ be the graph of $\mylns_T$ and $G'$ be its closure in $E \times M^m$.
Since $G \subseteq \Gamma(T)$ and $\Gamma(T)$ is closed in $E \times M^m$, the set $G'$ is also contained in the graph $\Gamma(T)$.
We can take $b \in M^n$ such that $(a,b) \in G' \setminus G$ because $\mylns_T$ is discontinuous at the point $a \in E$.
The point $b$ is contained in $T(a)$.

Set $\varepsilon =(|\!| b |\!|-|\!| \mylns_T(a) |\!|)/2$.
The definition of $\mylns_T$ yields that $\varepsilon >0$.
Due to the lower semi-continuity of $T$, there exists a definable open neighborhood $U$ of $a$ such that, for any $x \in U$, we have $|\!| y_x - \mylns_T(a) |\!|<\varepsilon$ for some $y_x \in T(x)$.
In particular, we get $|\!|\mylns_T(x)|\!| \leq |\!|y_x |\!| < |\!|\mylns_T(a)|\!|+\varepsilon = (|\!| b |\!|+|\!| \mylns_T(a) |\!|)/2$ for all $x \in U$.
On the other hand, the point $(a,b)$ is in the closure of the graph $G$ of $\mylns_T$.
We can take $x_0 \in U$ so that $ |\!| \mylns_T(x_0) -b  |\!| < \varepsilon$.
It implies that $ |\!| \mylns_T(x_0) |\!| \geq  |\!| b  |\!| -  |\!| \mylns_T(x_0) -b  |\!| >  |\!| b  |\!| -\varepsilon = (|\!| b |\!|+|\!| \mylns_T(a) |\!|)/2$.
It is a contradiction.
\end{proof}

The following lemma is a key lemma for the proof of Theorem \ref{thm:michael}.
\begin{lemma}\label{lem:michael_lns2}
Let $E$ be a definable, closed and bounded subset of $M^n$ and $T:E \rightrightarrows M^m$ be a definable lower semi-continuous set-valued map.
The least norm selection $\mylns_T:E \rightarrow M^m$ is definable.
\end{lemma}
\begin{proof}
We have only to prove that there exists a positive element $R \in M$ such that, for any $x \in E$, the intersection $T(x) \cap [-R,R]^m$ is not empty set.
If it is the case, consider the set-valued map $T':E \rightarrow M^m$ given by $T'(x)=T(x) \cap [-\sqrt{m}R,\sqrt{m}R]^m$.
The set $T'(x)$ is still closed and convex. 
We have $\mylns_{T}=\mylns_{T'}$ because there exists $y \in T(x)$ with $|\!| y |\!| \leq \sqrt{m}R$ for any $x \in E$ by the assumption.
The restriction $|\!| \cdot |\!| _R$ of the Euclidean norm $|\!| \cdot |\!|$ to $[-\sqrt{m}R,\sqrt{m}R]^m$ is definable because $\mathcal M$ has bounded definable multiplication compatible to the addition.
The map $\mylns_{T'}$ is defined by using the definable function $|\!| \cdot |\!| _R$.
Therefore, it is definable.
Consequently, the map $\mylns_{T}$ is also definable.

We demonstrate such an $R$ exists.
Assume the contrary.
For any positive $r \in M$, set 
\begin{align*}
&F_r=\{x \in E\;|\; T(x) \cap (-r,r)^m \neq \emptyset\}\text{ and }\\
&G_r=\{x \in E\;|\; T(x) \cap (-r,r)^m = \emptyset\}\text{.}
\end{align*}
We have $G_r \neq \emptyset$ for all $r>0$ by the assumption.
We show that $G_r$ is closed; equivalently, $F_r$ is open.
Take an arbitrary point $x_0 \in F_r$.
We can take $y_0 \in T(x_0) \cap  (-r,r)^m $.
Take a sufficiently small open box $V$ containing the point $y_0$ so that $V \subseteq (-r,r)^m$.
By lower semi-continuity, there exists an open neighborhood $U$ of $x_0$ such that $T(x) \cap V \neq \emptyset$ for all $x \in U$.
We get $T(x) \cap  (-r,r)^m \neq \emptyset$ for all $x \in U$.
It implies that $U$ is contained in $F_r$.

We obviously have $G_r \subseteq G_s$ when $r>s$.
The family $\{G_r\}_{r>0}$ is a definable filtered collection of nonempty definable closed sets.
Since $E$ is definably compact by Remark \ref{rem:def_compact}, the intersection $\bigcap_{r >0}G_r$ is not empty.
Take $x \in \bigcap_{r >0}G_r$, then $T(x)$ is an empty set by the definition of $G_r$.
It is a contradiction.
\end{proof}

We finally prove Theorem \ref{thm:michael}. 

\begin{proof}[Proof of Theorem \ref{thm:michael}]
Let $C$ be a quasi-special submanifold contained in $E$.
The notation $\Gamma(T|_C)$ denotes the graph of the restriction of $T$ to $C$.
The notation $\partial \Gamma(T|_C)$ denotes its frontier.
Let $\pi:M^{n+m} \rightarrow M^n$ be the projection onto the first $n$ coordinates.
Set $S=\pi(\partial \Gamma(T|_C))$.
We first demonstrate the following claim:
\medskip

\textbf{Claim 1.} $\dim S < \dim C$.
\medskip

Assume the contrary.
Set $d_0=\dim C$.
Since $S$ is contained in the closure of $C$, we have $\dim S=d_0$.
We also have $\dim (S \cap C)=d_0$ because $\dim \partial C < \dim C$ by Proposition \ref{prop:dim}(8).
Take a definable map $f:S \rightarrow M^{m}$ such that $(x,f(x)) \in \partial \Gamma(T|_C)$ for all $x \in S$ by the definable choice lemma \cite[Lemma 3.1]{Fuji5}.
We have $d_{T(x)}(f(x))>0$ because $T(x)$ is closed.
The definable subset $D$ of $S \cap C$ of points at which both $f$ and the map given by $x \mapsto d_{T(x)}(f(x))$ are continuous is of dimension $d_0$ by Proposition \ref{prop:dim}(7).
There exists a point $z \in D$  such that, for any open box $U$ containing the point $z$, the intersection $U \cap D$ is of dimension $d_0$ by Proposition \ref{prop:dim}(10).
Since $C$ is a quasi-special submanifold, we can take an open box $U_1$ containing the point $z$, a coordinate projection $\pi_1:M^n \rightarrow M^{d_0}$ and a definable continuous map $g: \pi_1(U_1) \rightarrow M^{n-d_0}$ such that $C \cap U_1$ is the graph of $g$.
The projection image $\pi_1(U_1 \cap D)$ has an interior by Proposition \ref{prop:dim}(11).
Therefore, shrinking $U_1$ if necessary, we may assume that $\pi_1(D)$ contains $\pi_1(U_1)$.
Move $z$ so that $x \in U_1 \cap D$.
Let $B$ be the intersection of $D$ with a closed box containing $z$ and contained in $U_1$.
We finally constructed the definable subset $B$ of $S \cap C$ such that
\begin{itemize}
\item $B$ is closed in $S$,
\item the interior of $B$ in $C$ is not empty, and 
\item the restrictions of $f$ and the map given by $x \mapsto d_{T(x)}(f(x))$ to $B$ are continuous.
\end{itemize}
We can take $N>0$ so that $d_{T(x)}(f(x))>N$ for all $x \in B$ by the max-min principle.
Take a point $x_0$ in the interior of $B$ in $C$ and set $\varepsilon=N/3$.
There exists an open neighborhood $U$ of $x_0$ in $S$ contained in $B$ such that $d(f(x),f(x_0))<\varepsilon$ for all $x \in U$ because $f$ is continuous.
Since $C$ is a quasi-special submanifold, the set $U$ is also an open neighborhood of $x_0$ in $C$, shrinking $U$ if necessary.
We can take $(x,y) \in \Gamma(T|_C)$ such that $x \in U$ and $d(y,f(x_0))<\varepsilon$ because $(x_0,f(x_0))$ is in the closure of $\Gamma(T|_C)$.
We get $d_{T(x)}(f(x)) \leq d(y,f(x)) \leq d(y,f(x_0))+d(f(x),f(x_0))< 2\varepsilon = 2N/3$.
It is a contradiction.
\medskip

\textbf{Claim 2.} For a given finitely many definable subsets $X_1,\ldots X_k$ of $E$, 
 there exists a partition $E=C_1 \cup \ldots \cup C_d$ of $E$ into quasi-submanifolds partitioning $X_1,\ldots X_k$ and satisfying the frontier condition such that the restriction $T|_{C_k}:C_k \rightrightarrows M^m$ to $C_k$ is continuous for any $1 \leq k \leq d$.
 \medskip 
 
We prove it by induction on $\dim E$.
When $E$ is of dimension zero, $E$ is discrete and closed by Proposition \ref{prop:dim}(1).
The set-valued map $T$ is continuous in this case.
The set $E$ is a quasi-special submanifold.
This case is easy.
We consider the case in which $\dim E>0$.
Apply Proposition \ref{prop:frontier_condition}.
We can get a partition $E=D_1 \cup \ldots \cup D_d$ into quasi-special submanifolds satisfying the frontier condition and partitioning $X_1,\ldots X_k$.
Permuting if necessary, we may assume that $D_1, \ldots, D_l$ are of dimension $\dim E$.
By Claim 1, for any $1 \leq i \leq l$, we can take definable subset $S_i$ of $D_i$ of dimension smaller than $\dim E$ such that the restriction of $T$ to $D_i \setminus S_i$ is continuous.
Apply the induction hypothesis to $\bigcup_{i=1}^l S_i \cup \bigcup_{i=l+1}^d D_i$ and the family $\{S_i\}_{1 \leq i \leq l} \cup \{D_i\}_{l<i \leq d}$.
We get the desired partition into quasi-special submanifolds.
\medskip

We finally prove the theorem by induction on the number of quasi-special submanifolds $d$ given in Claim 2.
The theorem immediately follows from Lemma \ref{lem:michael_lns} and Lemma \ref{lem:michael_lns2} when $d=1$.
We consider the case in which $d>1$.
We may assume that $\dim C_d=\dim E$ without loss of generality.
Set $D=\bigcup_{i=1}^{d-1}C_i$.
The frontier condition implies that $D$ is closed.
We can take a definable continuous selection $f_1:D \rightarrow M^m$ of the restriction $T|_D$ of $T$ to $D$ by the induction hypothesis.
Let $f_2:M^n \rightarrow M^m$ be a definable continuous extension of $f_1$ given by Theorem \ref{thm:weak_tietze}.
The map $f_3$ is its restriction to $E$.
We may assume that $f_3$ is constantly zero considering $T-f_3$ in place of $T$ by Lemma \ref{lem:michael_b}.

The least norm selection $\mylns_T$ is definable by Lemma \ref{lem:michael_lns2}.
The least norm selection $\mylns_{T|_{C_d}}$ for the restriction $T|_{C_d}$ coincides with the restriction of $\mylns_T$ to $C_d$ by the definition of least norm selections.
It is continuous by Claim 2 and Lemma \ref{lem:michael_lns}.
Consider the definable map $f:E \rightarrow M^m$ given by
$$
f(x)=\left\{
\begin{array}{ll} \mylns_T(x) & \text{ when }x \in C_d,\\ 0 & \text{ elsewhere.}\end{array}
\right.
$$
We have only to show that $f$ is continuous.
The map $f$ is obviously continuous out of the frontier $\partial C_d \cap E$ of $C_d$ in $E$.
Fix an arbitrary point $x_0 \in \partial C_d \cap E$.
We have $f(x_0)=0$.
Since $T$ is lower semi-continuous, for any $\varepsilon>0$, there exists $\delta>0$ such that we can take $y_x \in T(x)$ with $|\!|y_x|\!|<\varepsilon$ for all $x \in E$ with $|\!|x-x_0|\!|<\delta$.
We have $|\!| \mylns_T(x) |\!| \leq |\!|y_x|\!|<\varepsilon$ by the definition of least norm selections.
It implies that $f$ is continuous.
\end{proof}

\section{Psudo-definable spaces}\label{sec:definable_space}
\subsection{Definition}
\begin{definition}[Psudo-definable space]
Let $\mathcal M=(M;<,\ldots)$ be an expansion of a dense linear order without endpoints.
The set $M$ has the topology induced from the order $<$.
The Cartesian product $M^n$ equips the product topology.
The topology of a subset of $M^n$ is the relative topology.

A pair $(S, \{\varphi_i:U_i \rightarrow U'_i\}_{i \in I})$ of a topological space and a finite family of homeomorphism is called a \textit{pseudo-$\mathcal M$-definable space} if 
\begin{itemize}
\item $\{U_i\}_{i \in I}$ is a finite open cover of $S$, 
\item $U'_i$ is a definable subset of $M^{m_i}$ for any $i \in I$ and,
\item the composition $(\varphi_j|_{U_i \cap U_j}) \circ (\varphi_i|_{U_i \cap U_j})^{-1}:\varphi_i(U_i \cap U_j) \rightarrow \varphi_j(U_i \cap U_j)$ is a definable homeomorphism for any $i \neq j$ whenever $U_i \cap U_j$ is not an empty set.
\end{itemize}
Here, the notation $\varphi_i|_{U_i \cap U_j}$ denotes the restriction of $\varphi_i$ to ${U_i \cap U_j}$.
We simply call $S$ a \textit{pseudo-definable space} when the structure $\mathcal M$ is clear from the context.
The topological space $S$ is called the \textit{underlying topological space} of the pseudo-definable space.
The family $\{\varphi_i:U_i \rightarrow U_i'\}_{i \in I}$ is called a \textit{definable atlas on $S$}.
We often write $S$ instead of $(S, \{\varphi_i:U_i \rightarrow U'_i\}_{i \in I})$ for short.

We consider an expansion of an ordered divisible abelian group $\mathcal M=(M;<,0,+,\ldots)$.
Given a topological space $S$, two definable atlases $\{\varphi_i:U_i \rightarrow U_i'\}_{i \in I}$ and $\{\psi_j:V_j \rightarrow V_j'\}_{j \in J}$ on $S$ are \textit{equivalent} if, for all $i \in I$ and $j \in J$,
\begin{itemize}
\item $\varphi_i(U_i \cap V_j)$ and $\psi_j(U_i \cap V_j)$ are open definable subsets of $U'_i$ and $V'_j$, respectively, and
\item the homeomorphisms $(\psi_j|_{U_i \cap V_j}) \circ (\varphi_j|_{U_i \cap V_j})^{-1}:\varphi_i(U_i \cap V_j) \rightarrow \psi_j(U_i \cap V_j)$ are definable whenever $U_i \cap U_j \neq \emptyset$.
\end{itemize}
The above relation is obviously an equivalence relation.

A subset $X$ of the pseudo-definable space $S$ is \textit{definable} when $\varphi_i(X \cap U_i)$ are definable for all $i \in I$. 
When two atlases $\{\varphi_i:U_i \rightarrow U_i'\}_{i \in I}$ and $\{\psi_j:V_j \rightarrow V_j'\}_{j \in J}$ of a topological space $S$ is equivalent, it is obvious that a subset of the pseudo-definable space $(S,\{\varphi_i\}_{i \in I})$ is definable if and only if it is definable as a subset of the pseudo-definable space $(S,\{\psi_j\}_{j \in J})$.

The Cartesian product of two pseudo-definable spaces is naturally defined.
A map $f:S \rightarrow T$ between pseudo-definable spaces is \textit{definable} if its graph is definable in $S \times T$.
Note that a definable set is naturally a pseudo-definable space.
\end{definition}

\begin{proposition}\label{prop:closure}
Consider an expansion of a dense linear order without endpoints $\mathcal M=(M;<,\ldots)$ and a pseudo-definable space $(S, \{\varphi_i:U_i \rightarrow U'_i\}_{i \in I})$.
Let $X$ be a definable subset of the pseudo-definable space.
We have $$\mycl(X)=\bigcup_{i \in I} \varphi_i^{-1}(\mycl(\varphi_i(X \cap U_i)))\text{.}$$
\end{proposition}
\begin{proof}
Routine. Omitted.
\end{proof}

\begin{definition}[Dimension]\label{def:dim}
Consider a definably complete locally o-minimal structure $\mathcal M$.
Let $S$ be a pseudo-definable space and $\{\varphi_i:U_i \rightarrow U_i'\}_{i \in I}$ be its definable atlas.
The dimension of a definable set $X$ is defined by $$\dim X=\max_{i \in I} \dim \varphi_i(X \cap U_i)\text{.}$$
\end{definition}

\begin{proposition}\label{prop:well_defined_dim}
Consider a definably complete locally o-minimal structure $\mathcal M$ and two pseudo-definable spaces $(S,\{\varphi_i:U_i \rightarrow U_i'\}_{i \in I})$ and $(S,\{\psi_i:V_j \rightarrow V_j'\}_{j \in J})$ having the same underlying topological space $S$ such that their definable atlases are equivalent.
The dimension of a definable set of $S$ in Definition \ref{def:dim} is independent of choice of a definable atlas of $S$.
\end{proposition}
\begin{proof}
We first demonstrate that $\dim \varphi_i(X \cap U_i \cap V_j)=\dim \varphi_j(X \cap U_i \cap V_j)$ for all $i \in I$ and $j \in J$.
The equality is obvious when $X \cap U_i \cap U_j$ is an empty set.
We next consider the other case.
The map $(\psi_j|_{X \cap U_i \cap V_j}) \circ (\varphi_j|_{X \cap U_i \cap V_j})^{-1}$ is definable homeomorphism between $\varphi_i(X \cap U_i \cap V_j)$ and $\psi_j(X \cap U_i \cap V_j)$.
We have $\dim \varphi_i(X \cap U_i \cap V_j)=\dim \varphi_j(X \cap U_i \cap V_j)$ by Proposition \ref{prop:dim}(6).

The above equality and Proposition \ref{prop:dim}(5) imply that
\begin{align*}
\max_{i \in I} \dim \varphi_i(X \cap U_i) &= \max_{i \in I} \dim \varphi_i(X \cap \bigcup_{j \in J} (U_i \cap V_j)) = \max_{i \in I, j \in J}\dim \varphi_i(X \cap U_i \cap V_j)\\
&=\max_{i \in I, j \in J}\dim \psi_j(X \cap U_i \cap V_j) = \max_{j \in J} \dim \psi_j(X \cap \bigcup_{i \in I} (U_i \cap V_j))\\
&=\max_{j \in J} \dim \psi_j(X \cap V_j)\text{.} 
\end{align*}
We have demonstrated the proposition.
\end{proof}

\subsection{Definably compact sets}

\begin{definition}
For a set $X$, a family $\mathcal F$ of subsets of $X$ is called a \textit{filtered collection} if, for any $B_1, B_2 \in \mathcal F$, there exists $B_3 \in \mathcal F$ with $B_3 \subseteq B_1 \cap B_2$. 

Consider an expansion of a dense linear order without endpoints $\mathcal M=(M;<,\ldots)$.
Let $X$ and $T$ be pseudo-definable spaces.
The parameterized family $\{S_t\}_{t \in T}$ of definable subsets of $X$ is called \textit{definable} if the union $\bigcup_{t \in T} \{t\} \times S_t$ is definable in $T \times X$.

A parameterized family $\{S_t\}_{t \in T}$ of definable subsets of $X$ is a \textit{definable filtered collection} if it is simultaneously definable and a filtered collection.

A definable space $X$ is \textit{definably compact} if every definable filtered collection of closed nonempty subsets of $X$ has a nonempty intersection.
This definition is found in \cite[Section 8.4]{J}.
\end{definition}

\begin{remark}\label{rem:def_compact}
Let $\mathcal M=(M; <, \ldots)$ be a definably complete expansion of a dense linear order without endpoints.
Consider a definable subset $X$ of $M^n$ equipped with the affine topology.
The definable set $X$ is definably compact if and only if it is closed and bounded.
The literally same proof as that for o-minimal structures in \cite[Section 8.4]{J} works. 

The image of a definably compact psudo-definable space under a definable continuous map $f:X \rightarrow Y$ between psudo-definable spaces is definably compact.
It immidiately follows from the definition.
\end{remark}

\begin{lemma}\label{lem:limitset3}
Consider a definably complete locally o-minimal expansion of an ordered group $\mathcal M=(M;<,0,+,\ldots)$.
Let $(X, \{\varphi_i:U_i \rightarrow U'_i\}_{i \in I})$ be a definably compact pseudo-definable space.
Then, $U'_i$ are bounded for all $i \in I$ or there exists a definable homeomorphism between a bounded interval and an unbounded interval.
\end{lemma} 
\begin{proof}
We construct a definable homeomorphism between a bounded interval and an  unbounded interval assuming that $U'_i$ is not bounded for some $i \in I$.
Fix such an $i$.
Let $M^m$ be the ambient space of $U'_i$.

We first prove the following claim:
\medskip

{\textbf{Claim 1.}} There are no unbounded definable subsets of $U'_i$ of dimension zero.
\medskip

We prove the claim.
We lead to a contradiction that there exists an unbounded definable subset $D$ of $U'_i$ of dimension zero.
The notation $B(0;r)$ denotes the open box $(-r,r)^m$ for $r>0$.
Consider the definable filtered collection $\{\mycl(\varphi_i^{-1}(D \setminus B(0;r)))\}_{r>0}$ of closed nonempty subsets of $X$.
Since $X$ is definably compact, we can take $z \in \bigcap_{r>0} \mycl(\varphi_i^{-1}(D \setminus B(0;r)))$.
Take $k \in I$ so that $z \in U_k$.
Set $Z_r=\varphi_k(\varphi_i^{-1}(D \setminus B(0;r)) \cap U_k)$ for all $r>0$.
The point $\varphi_k(z)$ is not contained in $Z_r$ for a sufficiently large $r>0$, but it is contained in the closure of $Z_r$ by Proposition \ref{prop:closure}.
It is a contradiction to the fact that a definable set of dimension zero is closed.
We have demonstrated the claim.
\medskip

We can take $1 \leq j \leq m$ so that the image $\pi_j(U'_i)$ is not bounded.
Here, the notation $\pi_j:M^m \rightarrow M$ denotes the projection onto the $j$-th coordinate.
We assume that $\pi_j(U'_i)$ is not bounded from above. 
The proof for the case in which $\pi_j(U'_i)$ is not bounded from below is similar.
Under this circumstance, the following claim holds true:

\medskip

{\textbf{Claim 2.}} Let $C$ be an unbounded definable subset of $\pi_j(U'_i)$ which is bounded from below.
There exists an $R \in M$ such that the interval $[R, \infty)$ is contained in $C$.
\medskip

We prove the claim.
Let $D$ be the set of the discrete points in $C$.
Set 
\begin{align*}
E&=\{(a,b) \in M^2\;|\; a<b \text{ and  the open interval }(a,b) \text{ is a maximal interval}\\
&\qquad \text{ contained in }C\}\text{.}
\end{align*}  
We finally set $F=D \cup \{(a+b)/2\;|\; (a,b) \in E\}$.
The definable set $F$ is an unbounded definable subset of $\pi_j(U')$ of dimension zero.
It is a contradiction to Claim 1.
We have demonstrated the claim.
\medskip

We take a definable map $\tau:\pi_j(U'_i) \rightarrow U'_i$ by the definable choice lemma.
Set 
\begin{align*}
C_1&=\{x \in M\;|\; x>0, \tau \text{ is continuous at }x\}\text{.}
\end{align*}
Since the map $\tau$ is continuous at any discrete point in $\pi_j(U'_i)$, the set $C_1$ is an unbounded definable subset of $\pi_j(U'_i)$ by Theorem \ref{thm:monotonicity}.
The definable set $C_1$ contains an interval of the form $[R_1,\infty)$ by Claim 2.

Set $Y_r=\mycl(\varphi_i^{-1}(\tau([r,\infty))))$.
Since it is a definable filtered collection of closed nonempty subsets and $X$ is definably compact, we can take a point $z \in \bigcap_{r  \geq R} Y_r$.
We can take $k \in I$ so that $z \in U_k$.
Take a bounded open definable neighborhood $V'$ of $\varphi_k(z)$ in $U_k'$ and set $V=\varphi_k^{-1}(V')$.
Set 
\begin{align*}
C_2&=\{x \in C_1\;|\; \varphi_i^{-1}(\tau(x)) \in V\}\text{.}
\end{align*}
It is an unbounded definable set.
There exists $R_2 \in M$ such that the interval $[R_2, \infty)$ is contained in $C_2$ by Claim 2.

Let $M^n$ be the ambient space of $U'_k$ and consider the definable map $\eta:[R_2, \infty) \rightarrow U'_k$ given by $\eta(x)=\varphi_k(\varphi_i^{-1}(\tau(x)))$.
Let $p_l:M^n \rightarrow M$ be the coordinate projection onto the $l$-th coordinate.
Set 
\begin{align*}
C_3&=\{x \in [R_2,\infty))\;|\; p_l \circ \eta \text{ is continuous and locally monotone at }x\\
&\quad \text{ for all }1 \leq l \leq n\}\text{.}
\end{align*}
Since the difference $[R_2,\infty) \setminus C_3$ is discrete by the strong local monotonicity, the definable set $C_3$ is unbounded. 
The definable set $C_3$ contains an interval of the form $[R_3,\infty)$ by Claim 2.
The restriction $p_l \circ \eta|_{(R_3,\infty)}$ of $p_l \circ \eta$ to $(R_3,\infty)$ is monotone by \cite[Proposition 3.1]{Fuji} for all $1 \leq k \leq n$.
Since $\eta$ is injective, the composition $p_l \circ \eta|_{(R_3,\infty)}$ is strictly monotone for some $1 \leq l \leq n$.
It is a definable homeomorphism between a bounded open interval and an unbounded open interval.
\end{proof}

\begin{remark}\label{rem:bounded_coord}
Consider an definably complete locally o-minimal expansion of an ordered group $\mathcal M=(M;<,0,+,\ldots)$.
Let $X$ be a definably compact pseudo-definable space.
By Lemma \ref{lem:limitset3}, there always exists a definable atlas $\{\varphi_i:U_i \rightarrow U'_i\}_{i \in I}$ of $X$ such that $U'_i$ are bounded for all $i \in I$.
\end{remark}

\subsection{Imbedding of psudo-definable spaces}
\begin{definition}
Let $\mathcal M=(M; <, \ldots)$ be a definably complete expansion of a dense linear order without endpoints.
Let $X$ be a definable Hausdorff space and $(a, b)$ an interval.
We call a definable map $\gamma:(a, b) \to X$  a definable curve in $X$.
For $x \in X$, we write $\gamma \to x$ if $\lim_{t \to b}\gamma(t)=x$.
A definable curve $\gamma \to x \in Y$ with definable $Y \subseteq X$ such that $\gamma (a, b) \subseteq Y$,
we say that $\gamma$ is completable in $Y$.
\end{definition}

We can prove the following lemma similarly to \cite[Chapter 6, Lemma 4.2]{vdD}.

\begin{lemma}\label{lem:kawa_lem1}
Consider a definably complete locally o-minimal expansion of an ordered group $\mathcal M=(M;<,0,+,\ldots)$.
Let $X$ be a definable Hausdorff space, $f:X \to M^k$ be a definable map and $x \in X$.
Then $f$ is continuous at the point $x$ if and only if
for each definable curve $\gamma \subseteq X$ with $\gamma \to x$, we obtain $f \circ \gamma  \to f(x)$.
\end{lemma}
\begin{proof}
We may assume that the definable curve $\gamma$ is continuous by Proposition \ref{prop:dim}(7).
If $f$ is continuous, then the condition is clear.

We prove the converse.
Since $\mathcal M$ is an expansion of ordered group, 
the space $M^k$ has a definable metric $d$.
Assume $f$ is discontinuous at $p$.
Then, for any $\epsilon>0$ and any neighborhood $V$ of $f(p)$,
we have 
$$\{x \in X\;|\; d(f(x), f(p)) \geq \epsilon\} \cap V \neq \emptyset\text{.}$$
Since $X$ is a psudo-definable space,
we can take an open set $U \ni p$ and  a homeomorphism $g:U \to W\subseteq M^m$ such that the set $\{y \in W \;|\; d(f(g^{-1}(y)), f(p)) \geq \epsilon\}$ is a definable set and 
it intersects every neighborhood of $g(p)$.
It implies that the definable set $\{d(y, g(p))\;|\;y \in W, d(f(g^{-1}(y)), f(p)) \geq \epsilon\}$ contains an open interval of the form $(0,\delta)$.
By the definable choice, there exists a definable curve $\gamma:(0, \delta) \to W$ such that $d(\gamma (t), p)=t$ and $f(\gamma (t)), f(p)) \geq \epsilon$ for all $0 <t<\delta$.
Consider the pullback $\gamma':(0, \delta) \to X$ given by $\gamma'(t)=g^{-1}\circ \gamma (t)$.
We get $\gamma' \to p$, but do not obtain $f(\gamma') \to f(p)$.
\end{proof}

We get the following theorem with a proof similar to \cite[Chapter 10, Theorem 1.8]{vdD} using Remark \ref{rem:bounded_coord}.
\begin{theorem}\label{thm:compact_immbedding}
Consider an definably complete locally o-minimal expansion of an ordered group $\mathcal M=(M;<,0,+,\ldots)$ having definable bounded multiplication $\cdot$ compatible to $+$.
Every regular definably compact pseudo-definable space is definably imbeddable into some $M^n$.
\end{theorem}
\begin{proof}
Let $X$ be a regular definably compact pseudo-definable space with definable atlas $\{h_i:U_i \to V_i \subset M^{n_i}\}_{i=1}^k$.
We may assume that $V_i$ are always bounded by Remark \ref{rem:bounded_coord}. 
We proceed by induction $k$.
The case where $k=1$ is trivial.
We may assume that $k=2$.
If we can  prove this case, the general case is proved similarly.

We define the definable sets as follows:
\begin{align*}
&V_1=h_1(U_1 \cap U_2),\ \  V_2=h_2(U_1 \cap U_2),\ \ B_1=h_1(\partial U_2),\ \ 
B_2=h_2(\partial U_1),\\
&B'_i=\{x \in M^{n_i}\;|\;\exists y \in B_{3-i}, \forall \epsilon_1, \epsilon_2>0, \exists z \in U_1 \cap U_2,\ \\
&\qquad \qquad [d(x, h_1(z))<\epsilon_1 \text{ and } d(y, h_2(z))<\epsilon_2]\} \ \ (i=1,2).\\
\end{align*}
\medskip

\textbf{Claim 1.}\ $d(x, B_i')>0$ for each $x \in V_i$,  $i=1, 2$.
\medskip

We show the claim the case in which $i=1$ and the case in which $i=2$ is symmetry.
Assume that $d(x, B'_1)=0$ for some $x \in V_1$.
Note that $h_1^{-1}(x) \in U_1$ and $U_1 \cap h_2^{-1}(B_2)=\emptyset$.
Then, we get $U_1 \cap \mycl( h_2^{-1}(B_2))=\emptyset$.
Since $X$ is regular,
there exist disjoint open neighborhoods $D$ of $h_1^{-1}(x) \in U_1$ and 
$E$ of $h_2(B_2)$ in $U_2$.
For a contradiction, we find an element in $D \cap E$.

Since $h_1(D)$ is open in $V_1$,
we can take an $\epsilon >0$ such that $B(x, \epsilon) \cap V_1 \subset h_1(D)$.
Because $d(x, B'_1)=0$, there exists $x' \in B_1'$ with $d(x, x')<\epsilon$.
By the definition of $B_1'$, there exists $y \in B_2$ for which there are points $z \in U_1 \cap U_2$
with $h_1(z)$ close to $x'$ and $h_2(z)$ close to $y$.
Since $h_2(E)$ is open neighborhood of $y$ in $V_2$,
there exists $z \in U_1 \cap U_2$ with $d(x, h_1(z))<\epsilon, h_2(z) \in h_2(E)$.
Thus, $z \in D \cap E$.
This proves the claim.
\medskip

Set $d_i(z)=\min\{d(h_i(z), B_i'),1\}$ for $z \in X$ and $i=1,2$.
We define the map $h:X \to M^{n_1+n_2+2}$ by
\begin{align*}
h(z)=\left\{
\begin{array}{ll}
(d_1(z), d_1(z) \cdot h_1(z), 0, \dots, 0) & \text{ for }z \in U_1 \setminus U_2,\\
(d_1(z), d_1(z)\cdot h_1(z), d_2(z), d_2(z) \cdot h_2(z)) & \text{ for } z \in U_1 \cap U_2,\\
(0, \dots, 0, d_2(z), d_1(z)\cdot h_2(z)) & \text{ for }z \in U_2 \setminus U_1.
\end{array}
\right.
\end{align*}
We assumed that $\mathcal M$ has definable bounded multiplication.
Since both $d_i$ and $h_i$ are bounded for $i=1,2$, the map $h$ is definable.

By Claim 1, we get the following for each $z \in X$.
\begin{itemize}
\item $h(z) \in U_1 \setminus U_2$ if and only if the last $(1+n_2)$ coordinates of $h(z)$ are equal to $0$,
\item $h(z) \in U_2 \setminus U_1$ if and only if the first  $(1+n_1)$ coordinates of $h(z)$ are equal to $0$ 
and 
\item $h$ is injective.
\end{itemize}
\medskip

\textbf{Claim 2.}\ The map $h$ is continuous.
\medskip

We check the criterion given in Lemma \ref{lem:kawa_lem1}.
Let $\gamma$ be a curve in $X$ with $\gamma \to z \in X$.
By restricting the domain of $\gamma$,
we may assume that $\gamma$ lies in either 
completely in $U_1 \setminus U_2$, completely in $U_1 \cap U_2$, or completely 
in $U_2 \setminus U_1$.

When $\gamma$ lies in completely in $U_1 \setminus U_2$, we also have $z \in U_1 \setminus U_2$,
since $z \in U_2$ implies that $\gamma$ lies  at least partly in $U_2$.
The definition of $h$ shows that $h(\gamma) \to h(z)$.
We can prove the claim similarly when $\gamma$ lies completely in $U_2 \setminus U_1$.

Consider the case in which $\gamma$ lies completely in $U_1 \cap U_2$.
If we also have $z \in U_1 \cap U_2$, the definition of $h$ implies that $h(\gamma) \to h(z)$.
Let $z \not\in U_1 \cap U_2$, say  $z \not\in U_1$.
Then $z \in \partial U_1$, $y:=h_2(z) \in B_2$.
Take $x \in M^n$ such that $h_1(\gamma) \to x$.
Since $h_2(\gamma) \to y \in B_2$,
it follows that $x \in B_1'$.
Hence, we get $h(\gamma) \to h(z)$.
\medskip

\textbf{Claim 3.}\ The map $h$ maps $X$ homeomorphically onto $h(X)$.
\medskip

Let $K$ be a definable closed subset of $X$.
Since $h$ is injective and Claim 2, 
it suffices to prove that $h(K)$ is closed in $h(X)$.

Let $z \in X$ with $h(z) \in \mycl(h(K)) \cap h(X)$.
It suffices to show $z \in K$.
Since $h(z) \in \mycl(h(K))$,
there exists a definable curve $\gamma$ in $K$ such that $h(\gamma) \to h(z)$.
We may assume that either $\gamma$ lies completely in $U_1 \setminus U_2$,
$\gamma$ lies completely in $U_1 \cap U_2$,
or  $\gamma$ lies completely in $U_2 \setminus U_1$.
We can also assume that the domain of $\gamma$ is $(0, \epsilon)$.

We first consider the case in which $\gamma$ lies completely in $U_1 \setminus U_2$.
Since $h(\gamma (t))=(d_1(\gamma (t)), d_1(\gamma (t))h_1(\gamma (t)), 0, \dots, 0) \to h(z)$
as $t \to \epsilon$,
the last $(1+n_2)$ coordinates of $h(z)$ are $0$, so that $z \in U_1 \setminus U_2$ and
$h(z)=(d_1(z), d_1(z)h_1(z), 0, \dots, 0)$.
Using Claim 1, we get $h_1(\gamma) \to h_1(z)$.
Hence, we have $\gamma \to z$, so $z \in K$ as $K$ is closed in $X$.
The proof is similar when $\gamma$ lies completely in $U_2 \setminus U_1$.

The remaining case is the case in which $\gamma$ lies completely in $U_1 \cap U_2$.
We have  $h(\gamma (t)) \to h(z)$ as $t \to \epsilon$,
where $$h(\gamma (t))=(d_1(\gamma (t)), d_1(\gamma (t)) \cdot h_1(\gamma (t)), d_2(\gamma (t)), d_2(\gamma (t)) \cdot h_2(\gamma (t)))\text{.}$$
If $z \in U_1$, then we get $h(z)=(d_1(z), d_1(z)h_1(z), \dots)$.
By Claim 1, we have $h_1(\gamma (t)) \to h_1(z)$ and $\gamma \to z$.
Hence, we have $z \in K$.
If $z \in U_2$, the similar argument proves it.
\end{proof}

\section*{Acknowledgment}
MF would like to express his gratitude to the department of mathematics, Hiroshima university for its provision of references.
WK appreciates Akito Tsuboi and Kota Takeuchi. The discussions with them inspired him in developing Section \ref{sec:property_a}.

\end{document}